\documentclass{article}

\usepackage{amsrefs,amssymb,amsmath,amsthm,amsfonts,epsfig,graphicx}
\usepackage[utf8]{inputenc}
\usepackage[T1]{fontenc}
\usepackage[english]{babel}
\usepackage{hyperref}

\hypersetup{
    colorlinks,
    citecolor=black,
    filecolor=black,
    linkcolor=black,
    urlcolor=black
}
\theoremstyle{plain} 
\newtheorem{theorem}{Theorem}[section]

\newtheorem{proposition}[theorem]{Proposition}

\newtheorem{remark}[theorem]{Remark}

\newtheorem{problem}[]{Problem}

\numberwithin{equation}{section}

\DeclareMathOperator{\dist}{dist}

\DeclareMathOperator{\clos}{clos}

\newcommand{\hpol}{h_{pol}}
\newcommand{\htop}{h_{top}}
\renewcommand{\epsilon}{\varepsilon}
\newcommand{\R}{\mathbb R}
\newcommand{\Z}{\mathbb Z}
\newcommand{\N}{\mathbb N}
\newcommand{\Suno}{{\mathbb S}^1}
\theoremstyle{plain}

\title{Polynomial Entropy and Expansivity}

\author{Alfonso Artigue, Dante Carrasco-Olivera and Ignacio Monteverde}

\begin{document}

\maketitle

\begin{abstract}
In this paper we study the polynomial entropy of homeomorphism on
compact metric space. We construct a homeomorphism on a compact
metric space with vanishing polynomial entropy that it is not
equicontinuous. Also we give examples with arbitrarily small
polynomial entropy. Finally, we show that expansive homeomorphisms
and positively expansive maps of compact metric spaces with infinitely
many points have polynomial entropy greater or equal than 1.
\end{abstract}

\section{Introduction}

The study of chaotic dynamical systems presents several difficult
problems as, for example, to estimate the topological entropy 
as can be seen in \cites{Bowen72,HaKa,W82,BS} and, with more recent approaches, in \cites{ACV,CMM15,AKU}. 
Some systems are known to have positive
topological entropy.
For instance, in \cite{Fathi} it is shown that every expansive
homeomorphism on a compact metric space of positive topological
dimension has positive topological entropy. This result is extended
in \cite{K93} for continuum-wise expansivity.

For dynamical systems with vanishing topological entropy the
\emph{polynomial entropy} becomes an interesting object to measure
the complexity of the orbit structure. The concept of polynomial
entropy, that is recalled in \S \ref{secPolEnt}, was considered in
\cite{JPM13} with a detailed analysis of the basic properties and
examples from Hamiltonian dynamics. In \cites{La121,La122} it is
considered the problem of finding Riemannian metrics with minimal
polynomial entropy.

It is easy to prove that the polynomial entropy of an equicontinuous
homeomorphism is zero. In fact, a stronger result is true: a
homeomorphism is equicontinuous if and only if it has bounded
complexity, see \cite{BHM}*{Proposition 2.2} and \cite{Morales2015}*{Corollary 1.4}. Recall that $f\colon M\to M$ is
\emph{equicontinuous} if for all $\epsilon>0$ there is $\delta>0$
such that if $x,y\in X$, $\dist(x,y)<\delta$ then
$\dist(f^n(x),f^n(y))<\epsilon$ for all $n\in\mathbb{Z}$. For an
equicontinuous homeomorphism of a compact metric space there is a
compatible metric that makes it an isometry. Such metric can be
defined as $\dist'(x,y)=\sup_{n\in\mathbb{Z}}\dist(f^n(x),f^n(y))$,
for all $x,y\in X$.
On the circle the converse is true (see \cite{La}*{Theorem 1}): if a
homeomorphism of the circle has vanishing polynomial entropy then it is
equicontinuous (i.e., conjugate to a rotation). In \S
\ref{secNonIsoNullEnt} we give an example of a homeomorphism of a
compact metric space with vanishing polynomial entropy that is not
equicontinuous. In \cite{Cas1}*{Theoreme 6.1}, for every $\alpha
\in(1,2)$ a subshift $\sigma$ with $h_{pol}(\sigma)=\alpha$ is
constructed; taking products, we obtain that every real number
larger than 1 is the polynomial entropy of an expansive
homeomorphism (note that subshifts are expansive). In \S \ref{secNonInt}  for all $\varepsilon>0$ we
construct a (non-expansive) homeomorphism $f$ of a compact metric with
$h_{pol}(f)\in(0,\epsilon)$; so, given $0<a<b$ there exists $g$ with
$h_{pol}(g)\in(a,b)$.

In \S \ref{secExpHomPosEnt} we show that the polynomial entropy of
an expansive homeomorphism on a compact metric space with infinitely
many points is greater or equal to 1. We also sketch the proof of the
corresponding result for positively expansive maps, extending
\cite{Morales2015}*{Theorem 1.2} where it is shown that the
complexity is unbounded.

\section{Basic properties of polynomial entropy}
\label{secPolEnt} 
In this section we recall the definitions of entropy that we will use. 
Also, we prove a simple criterion to obtain polynomial entropy equal or greater than 1.

Let $(X,\dist)$ be a compact metric space and
consider a homeomorphism $f\colon X\to X$. For $E\subset X$,
$n\in\mathbb{Z}$ and $\varepsilon>0$ we say that $E$
$(n,\varepsilon)$-{\it spans} with respect to $f$, if for each $y\in
X$ there is $x\in E$ such that $\dist(f^k(x),f^k(y))\leq\epsilon$ for
all $0\leq k\leq n-1$. Let $r_n(\epsilon)$ denote the minimum
cardinality of the set which $(n,\epsilon)$-spans. Let us consider
\[
\htop^{\varepsilon}(f) =\limsup_{n\to\infty}\frac{1}{n}\log
r_n(\varepsilon).
\]
Following \cite{Bowen72}, the \emph{topological entropy} of $f$ is
defined as
\[
 \htop(f) =\lim_{\varepsilon\to 0}\htop^{\varepsilon}(f).
\]
Analogously,
\[
\hpol^{\varepsilon}(f)  =  \limsup_{n\to\infty}\frac{1}{\log n}\log
r_n(\epsilon).
\]
According to \cite{JPM13}, the \emph{polynomial entropy} of $f$ is
defined by
\[
\hpol(f) = \lim_{\varepsilon\to 0}\hpol^{\varepsilon}(f).
\]
A set $F\subset X$ is $(n,\epsilon)$-\emph{separated} if for
different points $x,y\in F$ there is $0\leq k< n$ such that
$\dist(f^k(x),f^k(y))\geq\epsilon$. We denote by $s_n(\epsilon)$ the
maximal cardinality of an $(n,\epsilon)$-separated subset of $Y$.
The topological and polynomial entropies can be defined equivalently
in terms of separated sets (see \cites{BS,JPM13}). 
Note that the
subtle change of $n$ by $\log(n)$ (as in the classical topological
entropy), gives completely different values for the entropy of the system.


For $x\in X$ and $\varepsilon>0$ we denote by $B_\epsilon(x)$ the
open ball of radius $\varepsilon$ and centered at $x$. The following
result is a generalization of \cite{La}*{Proposition 2.1} and
\cite{JPM13}*{the end of the proof of Proposition 5}. We recall that
a point $x\in X$ is \emph{recurrent} if for all $\epsilon>0$ there
is $m\geq 1$ such that $f^m(x)\in B_\epsilon(x)$.

\begin{proposition}
\label{propRecu}
 If $\hpol(f)<1$ then for all $\epsilon>0$ there is $m>0$
 such that for all $x\in X$ there is $n\in\mathbb{Z}$ such that
 $f^n(x)\in B_\epsilon(x)$ and $0< n\leq m$.
 In particular, if $\hpol(f)<1$ then every point $x\in X$ is recurrent.
\end{proposition}

\begin{proof}
Arguing by contradiction, assume that $\hpol(f)<1$ and there is $\epsilon>0$ such
that for all $m\geq 1$ there is $x\in X$ for which $f^n(x)\notin
B_{\epsilon}(x)$ for all $n=1,2,\dots,m$. 
Define the set $F=\{x,f^{-1}(x),\dots,f^{-m}(x)\}$. 
Given two different points $u=f^{-j}(x),v=f^{-k}(x)\in F$ 
with $0\leq j<k\leq m$, we have that 
$f^k(v)=x$ and $f^k(u)=f^{k-j}(f^j(u))=f^{k-j}(x)$. 
As $0<k-j\leq m$ we conclude that $f^k(u)\notin B_\epsilon(x)$ and 
$\dist(f^k(u),f^k(v))>\epsilon$.
This proves that
$F$ is an $(m,\epsilon)$-separated set
with $m+1$ elements. Consequently, $s_m(\epsilon)\geq m+1$ and $\hpol(f)\geq 1$. This
contradiction finishes the proof.
\end{proof}

\section{Calculating the polynomial entropy}

In this section we will construct a family of examples depending on
a sequence $a_n$. For $a_n=e^{-n}$ we will obtain a homeomorphism
with vanishing polynomial entropy that is not equicontinuous. For
$a_n=1/n^c$ with $c>1$ we will obtain an example with non-integer
polynomial entropy.

We start with the general properties of our construction. Consider the
circle $\Suno=\R/\Z$  with the usual distance $\dist_S$. Given
$\{a_n\}_{n\geq 1}$ a decreasing sequence of positive real numbers
such that $a_n\rightarrow 0$, take $$M=\Suno\times (\{a_n\colon \
n\in\N\}\cup\{0\}),$$ equipped with the distance $\dist$ defined as
$$\dist((x_1,y_1),(x_2,y_2))=\max\{\dist_S(x_1,x_2),|y_2-y_1|\}.$$
The space $M$ is a countable union of circles. Let $f\colon M\to M$
be the function such that 
\begin{equation}
\label{ecuFuncf}
f(x,y)=(x+y,y). 
\end{equation}
In this way, $f$ acts as
a rotation of angle $a_n$ on each circle of $M$.

\begin{remark}
Observe that $f$ is not equicontinuous. This can be seen using the
following property: if $g\colon X\to X$ is equicontinuous, $p\in X$
is a fixed point of $g$ and $q_n\rightarrow p$ then
$\lim_{n\to+\infty}g^n(q_n)\rightarrow p$. In our
case, we have that there exists a sequence $q_n$ that converges to
$(0,0)$ but $f^n(q_n)\not\to (0,0)$. Then, the homeomorphism $f$ we
are considering is not equicontinuous.
\end{remark}

Given $N\in\N$ and $\epsilon>0$, we will find an
$(N,\epsilon)$-spanning set for $f$. We define
$$H=\min \{n\in\N: \ Na_n<\varepsilon\}$$
and
$$r=\left\lfloor \frac 1 \epsilon\right\rfloor+1.$$
We can take a set with $r$ elements $P=\{p_1, \ldots, p_r\}\subset
\Suno$ such that
$$\dist_S(p_i,p_{i+1})<\varepsilon$$ for all $i=1,\ldots r$.
Define $A(N,\epsilon)=P\times(\{a_n: \ n\leq H\}\cup\{0\})$.

\begin{proposition}
\label{propANepSpan}
 The set $A(N,\epsilon)$ is an $(N,\epsilon)$-spanning set for $f$ with
 \[
  \# A(N,\epsilon)=(\left\lfloor 1/ \epsilon\right\rfloor+1)(H+1).
 \]
\end{proposition}

\begin{proof}
Let $q=(q_1,q_2)\in M$. If $q_2\geq a_H$ then there exists
$p=(x,q_2)\in A(N,\epsilon)$ such that $\dist_S(x,q_1)<\epsilon$. As
$f$ is a rotation in $\Suno\times\{q_1\}$, then
$$\dist(f^k(p),f^k(q))=\dist_S(x,q_1)<\epsilon$$ for all $k\in\Z$ (in particular, if $0\leq k<N$).

If $q_2<a_H$ then $Nq_2<Na_H<\epsilon$. Then, there is an element
$p=(x,0)\in A(N,\epsilon)$ such that $\dist_S(x,q_1+kq_2)<\epsilon$
for all $k=0,\ldots, N-1$. Therefore,
$$
\begin{array}{ll}
\dist(f^k(x,0),f^k(q_1,q_2)) & =\dist((x,0),(q_1+kq_2,q_2))\\
&=\max\{\dist_S(x,q_1+kq_2),q_2\}<\epsilon
\end{array}
$$
for all $k=0,\ldots, N-1$. This proves that $A(N,\epsilon)$ is an
$(N,\epsilon)$-spanning set for $f$. The cardinality of
$A(N,\epsilon)$ is easily calculated from definitions.
\end{proof}

\subsection{Vanishing polynomial entropy}
\label{secNonIsoNullEnt} The next result gives an example of a
homeomorphism with vanishing polynomial entropy that is not
equicontinuous.

\begin{proposition}
\label{propH0} If $a_n=e^{-n}$ then $\hpol(f)=0$.
\end{proposition}

\begin{proof}
 In this case $H=\lceil-\log(\frac\varepsilon N)\rceil$.
 By Proposition \ref{propANepSpan} we have that $A(N,\epsilon)$ is an $(N,\epsilon)$-spanning set with
 \[
  \# A(N,\epsilon)=\left(\left\lfloor 1/ \epsilon\right\rfloor+1\right)(\lceil-\log(\varepsilon /N)\rceil+1)
 \]
 Then, for all $\varepsilon>0$ it holds that
$$\lim_{N\to\infty} \frac{\log(\# A(N,\epsilon))}{\log(N)}=
\lim_{N\to\infty} \frac{\log(\left\lfloor
1/\epsilon\right\rfloor+1)+\log( \lceil\log(N)-\log(\epsilon)
\rceil+1)}{\log(N)}=0.$$ This proves that $\hpol(f)=0$.
\end{proof}

The example given in Proposition \ref{propH0} is defined in the
space $M$ that is a countable union of circles. In particular, the
space is not connected.

\begin{problem}
If $f$ is a homeomorphism of a connected compact metric space with
$\hpol(f)=0$ is it true that $f$ is equicontinuous?
\end{problem}

A homeomorphism $f$ is \emph{distal} if for all $x\neq y$ it holds
that
\[
 \inf_{n\in\Z}\dist(f^n(x),f^n(y))>0.
\]
Note that equicontinuous homeomorphisms are distal. As remarked
above, our homeomorphism $f$ is not equicontinuous. It is easy to
see that it is distal.

\begin{problem}
If $f$ is a homeomorphism of a compact metric space such that
$\hpol(f)=0$ must $f$ be distal?. It would be interesting to know a
dynamical characterization of those homeomorphisms with vanishing
polynomial entropy.
\end{problem}

\subsection{Non-integer polynomial entropy}
\label{secNonInt} Now we will show that the polynomial entropy may
not be an integer number.

\begin{remark}
In the case of the topological entropy we can easily construct a
homeomorphism with non-integer topological entropy from a
homeomorphism with non-vanishing integer topological entropy.
Suppose that $f\colon X\to X$ is a homeomorphisms. Consider $Y$ as a
disjoint union of $n$ copies of $X$. Formally,
$Y=\{1,\dots,n\}\times X$. Define $g\colon Y\to Y$ as
$g(i,x)=(i+1,x)$ if $i=1,\dots,n-1$ and $g(n,x)=(0,f(x))$. Then,
$g^n$ restricted to $\{i\}\times X$ is conjugate to $f$, for all
$i=1,\dots,n$. This implies that $\htop(g)=\htop(f)/n$. Then, if
$\htop(f)$ is an integer we can take $n\in\Z^+$ such that
$\htop(f)/n$ is not an integer and we obtain a homeomorphism
$g$ with non-integer topological entropy. For the polynomial entropy
we have that $\hpol(f^n)=\hpol(f)$ for all positive integer $n$.
This is proved in \cite{JPM13}*{Proposition 2}. Then, this trick does
not work to obtain a homeomorphism with non-integer polynomial
entropy.
\end{remark}

\begin{proposition}
\label{propHpolNoInt}
 If $a_n=\frac 1 {n^c}$, with $c\geq 1$, then $\frac 1 {c+1}\leq \hpol(f)\leq\frac 1 c$, 
 where $f\colon M\to M$ is the homeomorphism given by (\ref{ecuFuncf}).
 \end{proposition}
\begin{proof}
In this case $H=\lceil (\frac{N}\varepsilon)^\frac 1
c\rceil$ and for all $\varepsilon>0$ it holds that
$$\lim_{N\to\infty} \frac{\log(\# A(N,\epsilon))}{\log(N)}=
\lim_{N\to\infty} \frac{\log(\left\lfloor \frac 1
\epsilon\right\rfloor+1)+\log( \lceil (\frac{N}\varepsilon)^\frac 1
c\rceil)}{\log(N)}= \frac 1 c.$$
 This proves that $\hpol(f)\leq\frac 1 c$.

To prove the other inequality we will construct a separating set.
We can take a set $Y\subset \Suno$ with $r-1$ elements such that
$\dist_S(z,w)\geq\varepsilon$ for all $z,w\in Y$, $z\neq w$. Define
\[
 S(N,\epsilon)=Y\times\{a_n: n<D\},
\]
where $D=(\frac{cN}{\varepsilon})^{\frac 1{c+1}} $. Let $p=(x,\frac
1 {n^c})$ and $q=(y,\frac 1 {m^c})$ be different points of
$S(N,\varepsilon)$, with $n\geq m$. If $x\neq y$ then $\dist(p,q)\geq
\dist_S(x,y)\geq \varepsilon$. If $x=y$ then
\[
\begin{array}{ll}
\dist(f^k(p),f^k(q))& =\dist\left((x+\frac k{n^c},\frac 1{n^c}),(x+\frac k{m^c},\frac 1 {m^c})\right)\\
                    & \geq \dist_S\left(x+\frac {k}{n^c},x+\frac k{m^c}\right).
\end{array}
\]
Applying the mean value theorem we obtain
$$\dist_S\left(x+\frac {k}{n^c},x+\frac k{m^c}\right)=\left|\frac {k}{n^c}-\frac k{m^c}\right|=
 k\left|\frac {1}{n^c}-\frac 1{m^c}\right|=k\,c\,\theta^{c-1}\left|\frac 1 n -\frac 1 m\right|$$
for some $$\frac 1 n<\theta<\frac 1 m.$$ As $n<D$, we have for some
$k\leq N$:
$$\dist(f^k(p),f^k(q))> k\, c\, \left(\frac 1 n\right)^{c-1}\frac{m-n}{mn}\geq \frac {kc}{n^{c+1}}>\varepsilon.$$
We conclude
$$\lim_{N\to\infty} \frac{\log(\# S(N,\epsilon))}{\log(N)}=
\lim_{N\to\infty} \frac{\log(\left\lfloor \frac 1
\epsilon\right\rfloor)+\log( \lfloor (\frac{cN}{\varepsilon})^{\frac
1{c+1}}\rfloor)}{\log(N)}= \frac 1 {c+1}.$$ Therefore
$\hpol(f)\geq\frac 1 {c+1}$ and the proof ends.
\end{proof}

\begin{remark}
Applying \cite{JPM13}*{Proposition 2} we know that if $f$ and $g$ are
homeomorphisms of compact metric spaces then $\hpol(f\times
g)=\hpol(f)+\hpol(g)$. This and Proposition \ref{propHpolNoInt} give
us that the set $\{\alpha\in\R: \alpha=h(f)$ for some homeomorphism
$ f\}$ is dense in $\R^+$.
\end{remark}

\begin{problem}
 Is every positive real number the polynomial entropy of a homeomorphism of a compact metric space?
\end{problem}

\begin{problem}
If $f$ is a homeomorphism of a connected compact metric space with
$\hpol(f)$ finite, is it true that $\hpol(f)$ is an integer number?
\end{problem}

\begin{remark}
The example of Proposition \ref{propHpolNoInt}
 shows that distality does not imply vanishing polynomial entropy.
\end{remark}

\section{Expansive systems}
\label{secExpHomPosEnt}
In this section we will show that expansive homeomorphisms and
positively expansive maps of compact metric spaces with infinitely many points
have positive polynomial entropy. These results are true and well
known in the case of topological entropy if the space has positive
topological dimension.

\subsection{Expansive homeomorphisms}
A homeomorphism $f\colon X\to X$ of a compact metric space $(X,\dist)$ is \emph{expansive} if there
exists $\delta>0$ (an \emph{expansivity constant}) such that if
$\dist(f^n(x),f^n(y))<\delta$ for all $n\in\Z$ then $x=y$. 
It is known that for an expansive homeomorphism it holds that 
$h_{top}(f)=h_{top}(f,\delta)$, see \cite{BS}*{Proposition 2.5.7} (also \cite{Bowen72}).
We remark that $$\hpol(f)=\hpol(f,\delta)$$ 
also holds for the polynomial entropy, if $\delta$ is an expansivity constant.
The proof is analogous.

As we said, the topological entropy of an expansive
homeomorphism on a compact metric space of positive topological
dimension is positive, see \cites{Fathi,K93}. Also, there are
expansive homeomorphisms with vanishing topological entropy, for
example, expansive homeomorphisms of a countable spaces and the
non-wandering set of a Denjoy circle diffeomorphism. 
We will show that the polynomial entropy of every expansive homeomorphism of an 
infinite compact metric space is greater than or equal to $1$ (in particular, positive).

Let $\sigma\colon \Sigma\to\Sigma$ be a two-sided full shift on
$l\geq 2$ symbols. We say that $M\subset \Sigma$ is a \emph{minimal
set} if $M$ is the closure of the orbit of each $x\in M$.

\begin{theorem}
\label{teoHpoExp}
 If $f$ is an expansive homeomorphism of a compact metric space with infinitely many points then
 $\hpol(f)\geq 1$.
\end{theorem}

\begin{proof}
At first, we will prove the theorem in the case $f$ is the
restriction of a shift $\sigma\colon \Sigma\to\Sigma$ to a minimal
set $M$.

Let $\delta$ be an expansivity constant for $\sigma$, $x\in\Sigma$ be
a not periodic point for $\sigma$. From \cite{MoHe}*{Theorem 7.3} (or
\cite{Fer}*{Proposition 2}) we know that for every $n\geq 1$ there
exists $C\subset \Sigma$ with $n+1$ elements, included in the orbit
of $x$, such that if $a,b\in C$ with $a\neq b$ then
$\dist(\sigma^i(a),\sigma^i(b))>\delta$ for some $0\leq i\leq n$.
Now, taking $x\in M$, we can conclude that for
 all $\epsilon\leq\delta$ and for all $n\geq 1$
 there exists an $(n, \epsilon)$-separated set included in $M$ with $n+1$
 elements. So, $r_n(\epsilon)>n$ for all $\epsilon\leq\delta$ and for all $n\geq
 1$. From the definition, this implies that $h_{pol}(\sigma|_M)\geq
 1$.

 Now we will see the general case, that is, when the domain $X$ of $f$ is any compact metric space.
 Arguing by contradiction assume that $h_{pol}(f)<1$.
Proposition \ref{propRecu} implies that every point is recurrent.

Suppose that $p\in X$ is a periodic point, say $f^m(p)=p$. We will
show that there is $r>0$ such that $B_r(p)=\{p\}$. If this is not
the case we can take $x_n\to p$ with $x_n\neq p$ for all $n\geq 1$.
Since $f^m$ is expansive there is $\delta_0$ such that for all
$n\geq 1$ there is $k_n\in\Z$ with
\begin{equation}
 \label{ecuSaleBola}
f^{mk_n}(x_n)\notin B_{\delta_0}(p).
\end{equation}
Taking a subsequence we can suppose that $k_n\geq 0$ for all $n\geq
1$ (the other case is similar). We can also assume that $k_n$ is the
first number satisfying (\ref{ecuSaleBola}), that is,
$f^{mj}(x_n)\in B_{\delta_0}(p)$ for all $j=0,1,\dots, k_n-1$.
Taking a subsequence we can suppose that $f^{m(k_n-1)}(x_n)\to x$.
The continuity of $f$ implies that $f^{-mj}(x)\in
\clos(B_{\delta_0}(p))$ for all $j\geq 0$. Expansivity implies that
$f^{-mj}(x)\to p$ as $j\to+\infty$. This contradicts that $x$ is
recurrent and proves that every periodic point is an isolated point
of the space.

From \cite{Ma} we know that every minimal subset of $X$ is conjugate
to a subshift. As we know the theorem is true in this case, we
conclude that every minimal subset of $X$ is a periodic orbit. Since
every point is recurrent, we conclude that every point is periodic,
and then, every point is isolated. As the space is
compact, it must be finite. This contradiction proves the result.
\end{proof}

\begin{remark}
For a Sturmian minimal subshift it holds that $r_n(M,\delta)= n+1$ and
its polynomial entropy equals 1. In \cite{Cas} it is shown an
example of a subshift with vanishing topological entropy and
infinite polynomial entropy.
\end{remark}

\subsection{Positively Expansive Maps}

Let $f\colon X\to X$ be a continuous map of a compact metric space
$(X,\dist)$. We say that $f$ is \emph{positively expansive} if there is
$\delta>0$ such that if $\dist(f^n(x),f^n(y))<\delta$ for all $n\geq
0$ then $x=y$.

\begin{remark}
\label{rmkEntPosExp}
 If $f\colon X\to X$ is a positively expansive map and the topological dimension of $X$ is
 positive then $\htop(f)>0$, 
 hence $\hpol(f)=\infty$.
 This can be proved with the techniques of \cites{Fathi,K93}. 
\end{remark}

We recall that $X$ has positive topological dimension if and only if it is not totally disconnected.
In \cite{Morales2015}*{Theorem 1.2} it is shown that for every
positively expansive map there is $\epsilon>0$ such that
$\lim_{n\to+\infty}s_n(X,\epsilon)=+\infty$. We extend this result
as follows.

\begin{theorem}
\label{teorema-tipo_Kato}
 If $f\colon X\to X$ is a positively expansive map and $card(X)=\infty$ then $\hpol(f)\geq 1$.
\end{theorem}

\begin{proof}[Sketch of the proof]
As explained in Remark \ref{rmkEntPosExp}, if $X$ has positive dimension then $\hpol(f)=\infty$. 
If $X$ is totally disconnected it is known that $f$ is conjugate with a one-sided subshift. 
Therefore, the proof follows applying 
\cite{MoHe}*{Theorem 7.3} as in the proof of Theorem \ref{teoHpoExp}.
\end{proof}

\begin{center}
\small\small\thanks{Acknowledgements}
\end{center}
We thank Mauricio Achigar and José Vieitez for useful conversations
during the preparation of the paper. The second author would like to
thank the Departamento de Matem\'atica y Estad\'\i stica del Litoral at the
Universidad de la Rep\'ublica, Uruguay, for partial support and for
its hospitality.

\footnotesize
Alfonso Artigue and Ignacio Monteverde\\
Departamento de Matem\'atica y Estad\'{\i}stica del Litoral\\
Universidad de la Rep\'ublica\\
Gral. Rivera 1350, Salto, Uruguay\\
{\em{e-mail address:}} artigue@unorte.edu.uy, ignacio@cmat.edu.uy
\vspace{10pt}

\noindent Dante Carrasco-Olivera\\
Grupo de Investigaci\'on en Sistemas Din\'amicos y Aplicaciones\\
Departamento de Matem\'atica\\
Facultad de Ciencias\\
Universidad del B\'{\i}o B\'{\i}o\\
Av. Collao 1202, Casilla 5-C, Concepci\'on, Chile\\
{\em{e-mail address:}} dcarrasc@ubiobio.cl \vspace{10pt}

 \end{document}